\documentclass[12pt]{amsart}
\usepackage{amssymb,amsmath,amsthm,latexsym,color}
\usepackage{amscd}
\input amssym.def
\input amssym
\newtheorem{theorem}{Theorem}[section]

\newtheorem{remark}{Remark}[section]
\newtheorem{defi}{Definition}[section]
\newtheorem{prop}{Proposition}[section]

\newcommand{\be}{\begin{equation}}
\newcommand{\ee}{\end{equation}}

\renewcommand{\theequation}{\thesection.\arabic{equation}}
\renewcommand{\thetheorem}{\thesection.\arabic{theorem}}
\renewcommand{\theequation}{\thesection.\arabic{equation}}
\begin{document}

\title[] {On Replacement axioms for the Jacobi identity for vertex
algebras and their modules}

\author{Thomas J. Robinson}

%\thanks{}

\begin{abstract}
We discuss the axioms for vertex algebras and their modules, using
formal calculus.  Following certain standard treatments, we take the
Jacobi identity as our main axiom and we recall weak commutativity and
weak associativity.  We derive a third, companion property that we
call ``weak skew-associativity.''  This third property in some sense
completes an $\mathcal{S}_{3}$-symmetry of the axioms, which is
related to the known $\mathcal{S}_{3}$-symmetry of the Jacobi
identity.  We do not initially require a vacuum vector, which is
analogous to not requiring an identity element in ring theory.  In
this more general setting, one still has a property, occasionally used
in standard treatments, which is closely related to skew-symmetry,
which we call ``vacuum-free skew-symmetry.''  We show how certain
combinations of these properties are equivalent to the Jacobi identity
for both vacuum-free vertex algebras and their modules.  We then
specialize to the case with a vacuum vector and obtain further
replacement axioms.  In particular, in the final section we derive our
main result, which says that, in the presence of certain minor axioms,
the Jacobi identity for a module is equivalent to either weak
associativity or weak skew-associativity.  The first part of this
result has previously appeared and been used to show the (nontrivial)
equivalence of representations of and modules for a vertex algebra.
Many but not all of our results appear in standard treatments; some of
our arguments are different {}from the usual ones.
\end{abstract}

\maketitle

\renewcommand{\theequation}{\thesection.\arabic{equation}}
\renewcommand{\thetheorem}{\thesection.\arabic{theorem}}
\setcounter{equation}{0} \setcounter{theorem}{0}
\setcounter{section}{0}

\section{Introduction}
This paper gives an enhancement of certain axiomatic treatments of the
notion of module for a vertex algebra, and the notion of vertex
algebra itself.  We note especially that we handle certain issues of
the module theory that are more subtle than in the algebra theory
alone, a point made clear in \cite{Li1} (cf. \cite{LL}); we shall
discuss these issues in detail below.  The notion of vertex algebra
was first mathematically defined and considered by Borcherds in
\cite{B}.  Our treatment follows the formal calculus approach, which
originally appeared in \cite{FLM2} and was further developed in
\cite{FHL}.  In particular, the Jacobi identity, implicit in
Borcherds' definition, first appeared in \cite{FLM2}.  The original
mathematical motivation for the formulation of the notion of vertex
algebra and its variant notion of vertex operator algebra was related
to work done to construct a natural ``moonshine'' module for the
Monster group, a module conjectured to exist by J. McKay and
J. Thompson and constructed in \cite{FLM1} and \cite{FLM2}.  It was
soon recognized that vertex operator algebras were essentially
equivalent to chiral algebras in conformal field theory and string
theory, as was discussed in \cite{FLM2}.

This entire work is concerned with axiomatic issues.  It is well known
that in vertex algebra theory the Jacobi identity is very useful and
the most natural main axiom, but that it is often also natural and
convenient to use certain ``pieces'' of it as crutches to build back
up to the full story.  Of course, these ``pieces'' on their own do not
tell the whole story except when they have already essentially been
``put back together'' to yield the entire Jacobi identity.  Therefore
various other important strong properties which are partial
replacements for the Jacobi identity have been used prominently in the
theory.  For instance, in \cite{FLM2}, the authors first prove a
certain associativity property of lattice vertex operators, en route
to constructing certain families of vertex operator algebras (see
Chapter 8, in particular (8.4.32) and Remark 8.5.2).  In \cite{DL} and
\cite{Li1} ``weak commutativity'' (which we shall discuss below) was
found useful.  Further, in \cite{H1}, \cite{H2} and \cite{H3} a
certain associativity property for intertwining operators was proved
and developed (one which is much more difficult and deeper than the
associativity properties of vertex algebras).  We emphasize here that
the discussion in this paragraph is only for philosophical purposes
and that intertwining operators are well beyond the scope of this
paper.  These issues concerning the axioms reflect one of the
non-classical ingredients of vertex algebras, namely, the non-trivial
mathematical theory of the axioms themselves.  It is only once this
mathematical theory of the axioms can be handled that one gets to
analogues of more classical types of results such as the
representation theory and, in fact, without first developing this
mathematical theory of axioms one cannot even see the equivalence of
modules and representations.  This paper deals only with this
particular non-classical piece of vertex algebra theory, that is, the
mathematics of the axioms.  In particular, we shall show a slightly
new way to obtain a certain crucial result that says that in the
presence of certain minor axioms, the Jacobi identity for a module for
a vertex algebra is equivalent to weak associativity for a module for
a vertex algebra.  This result brings one (nearly) exactly to the
point where the analogues of more classical types of results can be
developed.  For instance, with this result known, one can immediately
develop the representation theory of vertex algebras as in \cite{Li1}
(cf. \cite{LL}), along lines analogous to the classical representation
theories, such as for Lie algebras.  An odd feature of this paper,
then, since we only deal with non-classical aspects of the theory, is
that we shall not have any examples corresponding to classical types
of examples.  However, the reader may regard the axioms themselves as
``non-classical examples.''  For another development of the
mathematics of the axioms of vertex algebras, a development extending
the one discussed here in a qualitatively different fashion, see
\cite{R}.

We take as our main axiom of vertex algebra the Jacobi identity, as in
\cite{FLM2} and \cite{FHL}.  It is well known that there are various
replacement axioms for the Jacobi identity that are useful in the
module and representation theory of and construction of vertex
algebras.  These replacement axioms are based on ``commutativity'' and
``associativity'' properties, as developed in \cite{FLM2}, \cite{FHL},
\cite{DL}, \cite{Li1}.  This theory is treated in detail in \cite{LL}.
In particular, each of the notions of weak commutativity and weak
associativity together with other more minor properties may replace
the Jacobi identity.  (We shall eventually be precise about what we
mean by both the terms ``minor axiom'' and ``minor property'' in the
context of this paper.  See Remark \ref{rem:minordef}).  For
instance, weak commutativity, as well as the equivalence of weak
commutativity together with certain minor axioms and the Jacobi
identity, first appeared in \cite{DL}, in the setting of vertex
operator algebras as well as in the much more general settings of
generalized vertex algebras and abelian intertwining algebras.  In the
case of vertex operator algebras, this equivalence was then
generalized in \cite{Li1} (cf. \cite{LL}) to handle the theory that
does not require any gradings of the algebras and also to handle
certain subtle and important issues concerning modules.  In this paper
we also work in a setting without gradings and also discuss certain of
these issues concerning modules; however, we do not handle the vertex
superalgebra case (which is a mild generalization).

Our purpose in this paper is twofold.  First, we introduce the notion
of ``weak skew-associativity'' to complement the properties of weak
commutativity and weak associativity of a vertex algebra
(cf. \cite{LL}).  This third property brings out the more fully the
$\mathcal{S}_{3}$-symmetric nature of the axioms for a vertex algebra,
which is suggested by the $\mathcal{S}_{3}$-symmetry of the Jacobi
identity presented in \cite{FHL}.  Just as weak commutativity and weak
associativity may be thought of as vertex-algebraic analogues of the
relations $a(bc)=b(ac)$ and $a(bc)=(ab)c$, respectively, for
commutative associative algebras, weak skew-associativity is analogous
to the third relation in a natural triangle: $b(ac)=(ab)c$.  We take
especial note that in each of these analogues, ``$c$'' always appears
in the rightmost position, a point of importance for the module
theory, as is discussed in Section 3.6 of \cite{LL} and which is
related to the main motivation for this paper.  We show how using weak
skew-associativity we may simplify certain proofs of the equivalence
of axiom systems for a vertex algebra and for a module for a vertex
algebra. In particular, in the final section we derive our main
result, which says that, in the presence of certain minor axioms, the
Jacobi identity for a module is equivalent to either weak
associativity or weak skew-associativity.  The equivalence of the
Jacobi identity (for a module) with weak associativity (for a module)
was shown in \cite{Li1} (cf. Theorem 4.4.5 in \cite{LL}) and enters
into the proof of the (nontrivial) equivalence of the notions of
representation of, and of module for, a vertex algebra (\cite{Li1};
cf. Theorem 5.3.15 in \cite{LL}).

We note that a formula closely related to weak skew-associativity
appeared in \cite{FLM2} in Remark 8.8.12, where those authors recorded
a certain iterate-type formula.  In addition, a more recent study of
certain ``formal'' types of axioms (``formal'' in the sense of
``formal commutativity'' which we discuss below) arose in \cite{DLM}
(cf. \cite{D}).  One of these axioms, Theorem 3.4 in \cite{DLM}
(cf. \cite{D} Theorem 4.2), bears a striking resemblance to weak
skew-associativity, or rather a ``formal'' type of skew-associativity
and it would be interesting to further study possible connections.
See also Theorem 3.8 in \cite{DLM} (cf. \cite{D} Theorem 4.5) for a
twisted generalization of this ``formal'' axiom.

We also note that the $\mathcal{S}_{3}$-symmetry referred to in the
last paragraph is really part of the mathematics of the axioms, the
non-classical part of vertex algebra theory.  The
$\mathcal{S}_{3}$-symmetry is a symmetry, not of vertex algebras, but
of the main axiom for a vertex algebra.  While it is true that
there is an analogue for this symmetry in Lie algebra theory, it is
simple enough there as to be dealt with in passing.  So while it is
technically non-classical in the sense which we have been using, it is
too easy to warrant an extensive independent study, whereas such types
of non-classical study seem to be necessary in the case of vertex
algebra theory.

Our second goal is to more fully check some of the dependencies
among the minor properties of a vertex algebra.  For instance, we
avoid using the vacuum vector in our considerations as long as
possible.  Since the vacuum vector is analogous to an identity
element, this approach is analogous to the study of rings without
identity, sometimes known as ``rngs.'' (However, we resist the
temptation to call these vacuum-free vertex algebras ``ertex
algebras'').  Although our motivation for this level of generality was
not example-driven, we refer the reader to \cite{BD} and \cite{HL},
where a vacuum-free setting appeared.

In Section 2, we set up some basic definitions and notation as
well as summarize certain formal calculus results which we need.
Almost all the relevant results may also be found in the fuller account
presented in \cite{FLM2} as well as in Chapter 2 of \cite{LL}.

In Section 3, we further develop the formal calculus, essentially
redoing many calculations which are usually performed after the
definition of vertex algebra is given.  Our goal is to systematize
these calculations and to demonstrate how they depend only on the
formal calculus rather than on any of the particulars of the vertex
algebras where they are applied.  We note that a similar approach was
taken in \cite{Li2}, where the reader should compare the statement of
Lemma 2.1 in \cite{Li2} with the statement of Proposition 3.3 of this
work (the proof of Lemma 2.1 in \cite{Li2} is related to the proofs of
both Propositions 3.2 and 3.3 in this work); the latter result is an
extension of the former.  In particular, this extension is relevant
for handling skew-associativity properties in addition to
commutativity and associativity properties.

In Section 4, we define a vertex algebra without vacuum, which
we call a vacuum-free vertex algebra.  We then note how our formal
calculus results in Section 3 may be immediately applied
without further comment to show how the main axiom, the Jacobi
identity, may be replaced in the definition of a vacuum-free vertex
algebra by any two of weak commutativity, weak associativity and weak
skew-associativity.  We next formalize what we call vacuum-free
skew-symmetry.  This notion is often used in the literature when
necessary but is not highlighted or named, mostly because with a
vacuum vector one is guaranteed to have a $\mathcal{D}$ operator and
therefore one may obtain skew-symmetry \cite{B} for a vertex algebra.
Since we are trying to work with a minimum of assumptions, we shall
not have such a $\mathcal{D}$ operator, at least at this stage in the
development, so that we cannot even state skew-symmetry.  We then show
how the Jacobi identity may be replaced as an axiom by vacuum-free
skew-symmetry together with any single one of weak commutativity, weak
associativity or weak skew-associativity.  The strategy employed is
the same as that used in \cite{LL}, where the analogy between vertex
algebras and commutative associative algebras provides classical
guides.

In Section 5, we define the notion of module for a vacuum-free
vertex algebra and show the parallel results concerning replacement of
the (module) Jacobi identity.

In Section 6, we recall the definition of vertex algebra (with
vacuum) so that we may exactly recover certain results considered in
\cite{LL}.  In developing the minor properties related to the vacuum
vector and the $\mathcal{D}$ operator, we make more prominent use of
vacuum-free skew-symmetry than in other treatments, as far as the
author is aware.  We then develop various replacement axioms for the
Jacobi identity and a couple of further minor results which will be
useful in the final section.  Whereas without the vacuum vector we
have derived consequential properties from the Jacobi identity by
``slicing'' it with residues or using visible symbolic symmetry, with
the vacuum vector at our disposal the strategy is to plug it into our
known formulas, thereby specializing them.  Then, as before, once we
have derived the remaining minor properties, we attempt to piece
combinations of them together, using classical guides when we can, in
order to build back up to recover certain remaining properties.  We
note that in \cite{Li2}, the author considered certain generalizations
of vertex algebras where, in place of the Jacobi identity, only weak
associativity was assumed as an axiom.  Because of this
generalization, it was natural (or really necessary) for the author to
examine more carefully certain dependencies.  Some of our results are
therefore, at least in aesthetic terms, developed in the same spirit.
In particular, the reader should compare Proposition 2.6 and Corollary
2.7 in \cite{Li2} with Propositions \ref{prop:WADDER} and
\ref{prop:WAplusSCgivesDB} of this work, where the former results were
already stronger than ours, as we discuss in Remark \ref{rem:Li2a}
(see also Remark \ref{rem:Li2b}).
                   
In Section 7, we define a module for a vertex algebra (with vacuum)
and, after a couple of preliminary results, we present the main result
of this paper.  Namely, we show that the module Jacobi identity may be
replaced by either module weak associativity or module weak
skew-associativity.  The result that the module Jacobi identity may be
replaced by module weak associativity was shown in \cite{Li1}
(cf. Theorem 4.4.5 in \cite{LL}) and a corollary to it was used in
\cite{LL} to show, following \cite{Li1}, the equivalence between the
notions of representation of and module for a vertex algebra (see
Theorem 5.3.15 in \cite{LL}).

There are many treatments of axiom systems for the notion of vertex
(operator) algebra in the literature, involving the results of
\cite{B}, \cite{FLM2}, \cite{FHL}, \cite{DL} and \cite{Li1} mentioned
above, but as far as we are aware, the results of the present paper
that did not essentially appear in those works have not appeared
before.  This paper is not intended as a survey of the many existing
treatments of axioms.  The reader may wish to consult the bibliography
in \cite{LL}.  In any case, by introducing weak skew-associativity and
viewing it as being on equal footing with weak commutativity and weak
associativity, we are bringing to light the fuller
$\mathcal{S}_{3}$-symmetric nature of the family of axiom systems,
extending beyond and suggested by the $\mathcal{S}_{3}$-symmetry of
the Jacobi identity \cite{FHL}.

Except for certain minor exceptions, the well-known results which we
shall recall appeared in \cite{B}, \cite{FLM2}, \cite{FHL}, \cite{DL}
and \cite{Li1}.  However, for convenience, we shall use the (mostly
expository) treatment in \cite{LL} when we give specific references to
basic results and definitions, etc.
   
I would like to thank Professor Haisheng Li who first taught me the
theory which is discussed and extended here and also Professor James
Lepowsky for many discussions.  In addition, many thanks to Professors
Yi-Zhi Huang and Haisheng Li for reading preliminary versions and
making helpful comments.  I would also like to thank the referee for
helpful comments.

\section{Formal calculus summarized}
\setcounter{equation}{0}

Here we recall some basic material set up in \cite{FLM2}; as we
mentioned above, we shall follow the expository treatment in
\cite{LL}.

We shall write $x,y,z,x_{1},x_{2},x_{3},\dots$ for commuting formal
variables.  In this paper, formal variables will always commute, and
we will not use complex variables.  All vector spaces will be over
$\mathbb{C}$, although one may easily generalize many results to the
case of a field of characteristic $0$.  Let $V$ be a vector space.  We
use the following:
\begin{align*}
V[[x,x^{-1}]]=\biggl\{ \sum_{n \in \mathbb{Z}}v_{n}x^{n}|v_{n} \in V  \biggr\}
\end{align*}
(formal Laurent series),
and some of its subspaces:
\begin{align*}
V((x))=\biggl\{ \sum_{n \in \mathbb{Z}}v_{n}x^{n}|v_{n} \in V, v_{n}=0
\text{ for sufficiently negative } n \biggr\}
\end{align*}
(truncated formal Laurent series),
\begin{align*}
V[[x]]=\biggl\{ \sum_{n  \geq 0}v_{n}x^{n}|v_{n} \in V \biggr\}
\end{align*}
(formal power series),
\begin{align*}
V[x,x^{-1}]=\biggl\{ \sum_{n \in \mathbb{Z}}v_{n}x^{n}|v_{n} \in V,
v_{n}=0 \text{ for all but finitely many } n \biggr\}
\end{align*}
(formal Laurent polynomials), and
\begin{align*}
V[x]=\biggl\{ \sum_{n \geq 0}v_{n}x^{n}|v_{n} \in V, v_{n}=0 \text{ for all
but finitely many } n \biggr\}
\end{align*}
(formal polynomials).  Often our vector space $V$ will be a vector
space of endomorphisms, $\text{\rm End}\,V$.  Even when $V$ is
replaced by $\text{\rm End}\,V$ some of these spaces are not
algebras, and we must define multiplication only up to a natural
restrictive condition.  This condition is the summability condition,
which is given in Definitions 2.1.4 and 2.1.5 in \cite{LL}.  In
general, when computing some series we shall say that it ``exists''
(in keeping with an analogy with analysis) when the coefficient of any
monomial has only finitely many contributing terms, that is, when the
coefficient of any monomial is finitely computable.  If the
coefficients are themselves endomorphisms of a vector space then we
only require that the coefficient of any monomial be finitely
computable when the series is applied to a fixed, but arbitrary,
vector (see Remark 2.1.3 in \cite{LL}).

Since some of our spaces, such as $\text{\rm End}\,V[[x,x^{-1}]]$, are
not algebras, but only have a partial multiplication, we are not
guaranteed that multiplication is associative even when it is defined.
In fact, multiplication is not associative in the usual sense, but
there is a replacement property.  If $F(x), G(x)$ and $H(x) \in
\text{\rm End}\,V[[x,x^{-1}]]$ and if the three products $F(x)G(x)$,
$G(x)H(x)$ and $F(x)G(x)H(x)$ all exist, then
\begin{align*}
(F(x)G(x))H(x)=F(x)(G(x)H(x))
\end{align*}
(See Remark 2.1.6 \cite{LL} and the preceding discussion).  We shall
refer to this replacement for associativity as ``partial associativity''.

\begin{remark} \rm
Throughout this paper, as in \cite{LL}, we often extend our spaces to
include more than one variable.  We state certain properties which
have natural extensions in such multivariable settings, which we will
also use without further comment.
\end{remark}

We define the operator $\text{Res}_{x}:V[[x,x^{-1}]] \rightarrow V$ by
the following: For $f(x)=\sum_{n \in \mathbb{Z}}a_{n}x^{n} \in
V[[x,x^{-1}]]$, 
\begin{align*}
\text{Res}_{x}f(x)=a_{-1}.
\end{align*}
Further, we shall frequently use the notation $e^{w}$ to refer to the
formal exponential expansion, where $w$ is any formal object for which
such expansion makes sense.  For instance, we have the linear operator
$e^{y\frac{d}{dx}}:\mathbb{C}[[x,x^{-1}]] \rightarrow
\mathbb{C}[[x,x^{-1}]][[y]]$:
\begin{align*}
e^{y\frac{d}{dx}}=\sum_{n \geq
0}\frac{y^{n}}{n!}\left(\frac{d}{dx}\right)^{n}.
\end{align*}
We have (see (2.2.18) in \cite{LL}), the \it{automorphism property}\rm:
\begin{align}
\label{defi:aut}
e^{y\frac{d}{dx}}(p(x)q(x))=\left(e^{y\frac{d}{dx}}p(x)\right)
\left(e^{y\frac{d}{dx}}q(x)\right),
\end{align}
for all $p(x) \in \text{\rm End}\,V[x,x^{-1}] $ and $q(x) \in
\text{\rm End}\,V[[x,x^{-1}]]$.  We use the \it{binomial expansion
convention}, \rm which states that
\begin{align}
\label{binexpconv}
(x+y)^{n}=\sum_{k \geq 0}\binom{n}{k}x^{n-k}y^{k},
\end{align}
where we allow $n$ to be any integer and where we define
\begin{align*}
\binom{n}{k}=\frac{n(n-1)(n-2) \cdots (n-k+1)}{k!};
\end{align*}
the binomial expression is expanded in nonnegative powers of the
second-listed variable.  We also have (see Proposition 2.2.2 in
\cite{LL}) the \it{formal Taylor theorem}\rm :
\begin{prop}
Let $v(x) \in V[[x,x^{-1}]]$. Then
\begin{align*}
e^{y\frac{d}{dx}}v(x)&=v(x+y).
\end{align*}
\end{prop}
\begin{flushright} $\square$ \end{flushright}

For completeness we include a proof of the following frequently used
fact, which equates two different expansions.
\begin{prop}
\label{prop:assocarith}
For all $n \in \mathbb{Z}$,
\begin{align*}
(x+(y+z))^{n}=((x+y)+z)^{n}.
\end{align*}
\end{prop}
\begin{proof}
If $w_{1}$ and $w_{2}$ are
commuting formal objects, then $e^{w_{1}+w_{2}}=e^{w_{1}}e^{w_{2}}$.
Thus we have
\begin{align*}
(x+(y+z))^{n}
=e^{(y+z)\frac{\partial}{\partial x}}x^{n}
=e^{y\frac{\partial}{\partial x}}\left(e^{z\frac{\partial}{\partial x}}x^{n}\right)
=e^{y\frac{\partial}{\partial x}}(x+z)^{n}
=((x+y)+z)^{n}.
\end{align*}
\end{proof}
We note as a consequence that for all integers $n$ (and not just
nonnegative integers) we have the (non-vacuous) fact that
\begin{align*}
((x+y)-y)^{n}=(x+(y-y))^{n}=x^{n}.
\end{align*}

We define the formal delta function by
\begin{align*}
\delta (x) = \sum_{ n \in \mathbb{Z}}x^{n}.
\end{align*}
We have (see Proposition 2.3.21 and Remarks 2.3.24 and 2.3.25 in
\cite{LL}) the delta-function substitution property:
\begin{prop}
\label{prop:deltasub}
For $f(x,y,z) \in \text{\rm End}\,V[[x,x^{-1},y,y^{-1},z,z^{-1}]]$ 
such that for each fixed $v \in V$
\begin{align*}
f(x,y,z)v \in \text{\rm End}\,V[[x,x^{-1},y,y^{-1}]]((z))
\end{align*}
and such that
\begin{align*}
\text{\rm lim}_{x \rightarrow y}f(x,y,z)
\end{align*}
exists (where the ``limit'' is the indicated formal substitution), we
 have
\begin{align*}
\delta\left(\frac{y+z}{x}\right)f(x,y,z)=\delta\left(\frac{y+z}{x}
\right)f(y+z,y,z)=\delta\left(\frac{y+z}{x}\right)f(x,x-z,z).
\end{align*}
\end{prop}
\begin{flushright} $\square$ \end{flushright}

As in \cite{LL}, we use similarly verified
substitutions below without comment.

\section{Formal calculus further developed}
\setcounter{equation}{0}

Certain elementary identities concerning delta functions are very
convenient for dealing with the arithmetic of vertex algebras and, in
fact, in some cases, are fundamental to the very notion of vertex
algebra.  We state and prove some such identities in this
section.

The following well-known proposition appears as Proposition 2.3.8 in
\cite{LL}.  (Again, see the original works mentioned above.)  
We present an alternate proof which is implicitly
exploiting the $\mathcal{S}_{3}$-symmetry underlying the notion of
vertex algebra.  We include this alternate proof to emphasize that
$\mathcal{S}_{3}$-symmetry is playing a role in the development of the
ideas in this paper, as we discussed in the introduction.  For a
precise formulation see Section 2.7 in \cite{FHL} and Section 3.7 in
\cite{LL}.

\begin{prop}
We have the following two elementary identities:
\begin{align}
\label{twotermdelta}
x_{1}^{-1}\delta\left(\frac{x_{2}+x_{0}}{x_{1}}\right)-
x_{2}^{-1}\delta\left(\frac{x_{1}-x_{0}}{x_{2}}\right)=0
\end{align}
and
\begin{align}
\label{threetermdelta}
x_{0}^{-1}\delta\left(\frac{x_{1}-x_{2}}{x_{0}}\right)-
x_{0}^{-1}\delta\left(\frac{-x_{2}+x_{1}}{x_{0}}\right)-
x_{1}^{-1}\delta\left(\frac{x_{2}+x_{0}}{x_{1}}\right)=0.
\end{align}
\end{prop}
\begin{proof}
First observe that 
\begin{align*}
y^{-1}\delta\left(\frac{x}{y}\right)=(x-y)^{-1}+(y-x)^{-1},
\end{align*}
where we note that a clue as to why this identity holds is that both
expressions are annihilated by $x-y$.  Then by the formal Taylor
theorem
\begin{align*}
y^{-1}\delta\left(\frac{x+z}{y}\right)&=e^{z\frac{d}{dx}}
y^{-1}\delta\left(\frac{x}{y}\right)\\
&=e^{z\frac{d}{dx}}((x-y)^{-1}+(y-x)^{-1})\\
&=((x+z)-y)^{-1}+(y-(x+z))^{-1}.
\end{align*}
Being careful with minus signs we may respectively expand all the
terms in the left-hand side of (\ref{twotermdelta}) in this manner yielding
\begin{align*}
((x_{2}+x_{0})-x_{1})^{-1}+(x_{1}-(x_{2}+x_{0}))^{-1}\\
-((x_{1}-x_{0})-x_{2})^{-1}-(x_{2}-(x_{1}-x_{0}))^{-1}.
\end{align*}
Now we get by Proposition \ref{prop:assocarith} that the first and
fourth, and the second and third terms pairwise cancel each other thus
giving us (\ref{twotermdelta}).
Similarly we may respectively expand all the terms in the left-hand side of
(\ref{threetermdelta}) to get
\begin{align*}
((x_{1}-x_{2})-x_{0})^{-1}+(x_{0}-(x_{1}-x_{2}))^{-1}\\
-((-x_{2}+x_{1})-x_{0})^{-1}-(x_{0}-(-x_{2}+x_{1}))^{-1}\\
-((x_{2}+x_{0})-x_{1})^{-1}-(x_{1}-(x_{2}+x_{0}))^{-1}.
\end{align*}
Now we get by Proposition \ref{prop:assocarith} that the first and
sixth terms, the third and fifth terms, and the second and fourth
terms pairwise cancel each other thus giving us (\ref{threetermdelta}).
\end{proof}

A slight variant of the following Proposition appeared in \cite{LL} as
Propositions 2.3.26 and 2.3.27:
\begin{prop}
\label{prop:forwardenhanceddelta}
Let $g(x_{0},x_{1},x_{2}) \in V[[x_{0},x_{1},x_{2}]]$.  Next, for
$a,b$ and $c \geq 0$, let
\begin{align*}
f(x_{0},x_{1},x_{2})=\frac{g(x_{0},x_{1},x_{2})}{x_{0}^{a}x_{1}^{b}x_{2}^{c}}.
\end{align*}
Then
\begin{align*}
x_{1}^{-1}\delta\left(\frac{x_{2}+x_{0}}{x_{1}}\right)f(x_{0},x_{2}+x_{0},x_{2})=
x_{2}^{-1}\delta\left(\frac{x_{1}-x_{0}}{x_{2}}\right)f(x_{0},x_{1},x_{1}-x_{0})\end{align*}
and
\begin{align*}
&x_{0}^{-1}\delta\left(\frac{x_{1}-x_{2}}{x_{0}}\right)f(x_{1}-x_{2},x_{1},x_{2})-
x_{0}^{-1}\delta\left(\frac{-x_{2}+x_{1}}{x_{0}}\right)f(-x_{2}+x_{1},x_{1},x_{2})-\\
&x_{1}^{-1}\delta\left(\frac{x_{2}+x_{0}}{x_{1}}\right)f(x_{0},x_{2}+x_{0},x_{2})=0.
\end{align*}
\end{prop}
\begin{proof}
We have, for instance, by the (partial) formal delta substitution
principle that
\begin{align}
\label{forwardenhanceddeltaproof}
x_{0}^{-1}\delta\left(\frac{x_{1}-x_{2}}{x_{0}}\right)f(x_{1}-x_{2},x_{1},x_{2})
&=x_{0}^{-1}\delta\left(\frac{x_{1}-x_{2}}{x_{0}}\right)f((x_{0}+x_{2})-x_{2},x_{1},x_{2})
\nonumber\\
&=x_{0}^{-1}\delta\left(\frac{x_{1}-x_{2}}{x_{0}}\right)f(x_{0},x_{1},x_{2}).
\end{align}
Similar substitutions on the other terms also leave the delta function
multiplied by the common factor, $f(x_{0},x_{1},x_{2})$.  Therefore
(\ref{twotermdelta}) and (\ref{threetermdelta}) yield the result.
\end{proof}
\begin{remark} \rm
In the proof of Proposition \ref{prop:forwardenhanceddelta} we
specifically chose the order in which to perform the substitutions,
rather than to reverse the equalities in
(\ref{forwardenhanceddeltaproof}), because it is easier to see that all
``existence'' type conditions are met.
\end{remark}
We also have the following converse proposition, which the author
believes has not previously appeared in full, but much of it may be
viewed as placing in a more elementary setting the relevant existing
arguments used in the axiomatic theory presented in \cite{LL}.  We
note that a very similar approach was taken in Lemma 2.1 \cite{Li2} where
the author had already proved, in a somewhat different manner, some of
the implications (related to the Jacobi identity and commutativity and
associativity properties, but not skew-associativity properties) of
the following proposition.

\begin{prop}
Let $f(x_{1},x_{2}),\,g(x_{1},x_{2}),$ and $h(x_{1},x_{2}) \in
V((x_{1}))((x_{2}))$.
We have certain implications among the following statements:
\begin{itemize}
\item
(A) We have 
\begin{align}
x_{0}^{-1}\delta\left(\frac{x_{1}-x_{2}}{x_{0}}\right)f(x_{1},x_{2})-
&x_{0}^{-1}\delta\left(\frac{-x_{2}+x_{1}}{x_{0}}\right)g(x_{2},x_{1})
\nonumber\\
&-x_{1}^{-1}\delta\left(\frac{x_{2}+x_{0}}{x_{1}}\right)h(x_{2},x_{0})=0.
\label{eq:A}
\end{align}
\item
(B)
There exists some $m_{1} \geq 0$ such that
\begin{align*}
(x_{1}-x_{2})^{m_{1}}(f(x_{1},x_{2})-g(x_{2},x_{1}))=0.
\end{align*}
\item
(C)
There exists some $m_{2} \geq 0$ such that
\begin{align*}
(x_{0}+x_{2})^{m_{2}}(f(x_{0}+x_{2},x_{2})-h(x_{2},x_{0}))=0.
\end{align*}
\item
(D)
There exists some $m_{3} \geq 0$ such that
\begin{align*}
(x_{1}-x_{0})^{m_{3}}(g(-x_{0}+x_{1},x_{1})-h(x_{1}-x_{0},x_{0}))=0.
\end{align*}
\item
(E) There exist $p_{1}(x_{1},x_{2}) \in V[[x_{1},x_{2}]]$
and $a_{1},b_{1},c_{1} \geq 0$ such that
\begin{align*}
f(x_{1},x_{2})=\frac{p_{1}(x_{1},x_{2})}{(x_{1}-x_{2})^{a_{1}}x_{1}^{b_{1}}x_{2}^{c_{1}}}
\end{align*}
\qquad and
\begin{align*}
g(x_{2},x_{1})=\frac{p_{1}(x_{1},x_{2})}{(-x_{2}+x_{1})^{a_{1}}x_{1}^{b_{1}}x_{2}^{c_{1}}}.
\end{align*}
\item
(F) There exist $p_{2}(x_{0},x_{2}) \in V[[x_{0},x_{2}]]$
and $a_{2},b_{2},c_{2} \geq 0$ such that
\begin{align*}
f(x_{0}+x_{2},x_{2})=\frac{p_{2}(x_{0},x_{2})}{x_{0}^{a_{2}}(x_{0}+x_{2})^{b_{2}}x_{2}^{c_{2}}}
\end{align*}
\qquad and
\begin{align*}
h(x_{2},x_{0})=\frac{p_{2}(x_{0},x_{2})}{x_{0}^{a_{2}}(x_{2}+x_{0})^{b_{2}}x_{2}^{c_{2}}}.
\end{align*}
\item
(G) There exist $p_{3}(x_{0},x_{1}) \in V[[x_{0},x_{1}]]$
and $a_{3},b_{3},c_{3} \geq 0$ such that
\begin{align*}
g(-x_{0}+x_{1},x_{1})=\frac{p_{3}(x_{0},x_{1})}{x_{0}^{a_{3}}x_{1}^{b_{3}}(-x_{0}+x_{1})^{c_{3}}}
\end{align*}
\qquad and
\begin{align*}
h(x_{1}-x_{0},x_{0})=\frac{p_{3}(x_{0},x_{1})}{x_{0}^{a_{3}}x_{1}^{b_{3}}(x_{1}-x_{0})^{c_{3}}}.
\end{align*}
\end{itemize}
Namely, we have:
\begin{align*}
(ia) \quad (A) \Rightarrow (B), &  \qquad \qquad (iia) \quad (B)
\Rightarrow (E), &  (iiia) \quad (E) \,\text{ \rm and }\, (F) \Rightarrow (A),\\
(ib) \quad (A) \Rightarrow (C), &   \qquad  \qquad (iib) \quad (C)
\Rightarrow (F), &  (iiib) \quad (E) \,\text{ \rm and }\,(G) \Rightarrow (A),\\
(ic) \quad (A) \Rightarrow (D), &  \qquad \qquad (iic) \quad (D)
\Rightarrow (G), \quad \text{   \rm   and }& (iiic) \quad (F) 
\,\text{ \rm and }\, (G) \Rightarrow (A).\\
\end{align*}
\label{prop:elem}
\end{prop}
\begin{remark} \rm
In \cite{LL} the authors essentially used only statements
(A), (B), (C), (E) and (F).  It is an application of the principles of
statements (D) and (G) which leads to the new notion of weak
skew-associativity.  
\end{remark}
\begin{remark} \rm
It is also easy to see that we also have some converse implications
from among the above list such as $(E) \Rightarrow (B)$, $(F)
\Rightarrow (C)$ and $(G) \Rightarrow (D)$.  The reader will also soon
see that there are various implications for this later, but we shall
not use these results in this paper so we shall make no further
mention of them.
\end{remark}
\begin{proof}
The proofs of (ia), (ib), and (ic) are similar.  For (ia) we note that
(A) trivially implies that the left-hand side of (\ref{eq:A}) is lower
truncated in powers of $x_{0}$.  Further, the third term in the
left-hand side of (\ref{eq:A}) is visibly lower truncated in powers of
$x_{0}$, and therefore the sum of the remaining two terms must be
lower truncated in powers of $x_{0}$, which precisely yields (B).  The
proofs of (ib) and (ic) similarly follow from the lower truncation, in
the left-hand side of (\ref{eq:A}), of $x_{1}$ and $x_{2}$
respectively, after the obvious (especially in light of the statements
of (C) and (D)) delta function substitutions are made.

The proofs of (iia), (iib) and (iic) are similar.  We show only (iia).  Since
$f(x_{1},x_{2}) \in V((x_{1}))((x_{2}))$ and since $g(x_{2},x_{1}) \in
V((x_{2}))((x_{1}))$ we must have that
$(x_{1}-x_{2})^{m_{1}}f(x_{1},x_{2})$ and
$(x_{1}-x_{2})^{m_{1}}g(x_{2},x_{1})$ are both $\in V((x_{1},x_{2}))$.
So there is some $p_{1}(x_{1},x_{2}) \in V[[x_{1},x_{2}]]$ and $b, c
\geq 0$ such that
\begin{align*}
(x_{1}-x_{2})^{m_{1}}f(x_{1},x_{2})=(x_{1}-x_{2})^{m_{1}}g(x_{2},x_{1})
=\frac{p_{1}(x_{1},x_{2})}{x_{1}^{b}x_{2}^{c}}.
\end{align*}
A careful application of partial associativity allows us to cancel
(not simultaneously!) the polynomial terms on the left-hand sides to
get the result (cf. Remarks 3.25 and 3.28 in \cite{LL}).

The proofs of (iiia), (iiib) and (iiic) are similar.  We show only (iiia).
That is, we assume the truth of statements $(E)$ and $(F)$ and prove
the truth of statement $(A)$.  We have
\begin{align*}
\frac{p_{2}(x_{0},x_{2})}{x_{0}^{a_{2}}(x_{0}+x_{2})^{b_{2}}x_{2}^{c_{2}}}
&=f(x_{0}+x_{2},x_{2})=e^{x_{2}\frac{\partial}{\partial
x_{0}}}f(x_{0},x_{2})\\ &= e^{x_{2}\frac{\partial}{\partial
x_{0}}}\frac{p_{1}(x_{0},x_{2})}{(x_{0}-x_{2})^{a_{1}}x_{0}^{b_{1}}x_{2}^{c_{1}}}\\
&=\frac{p_{1}(x_{0}+x_{2},x_{2})}{x_{0}^{a_{1}}(x_{0}+x_{2})^{b_{1}}x_{2}^{c_{1}}},
\end{align*}
where we used the consequence of Proposition \ref{prop:assocarith}.  We
use this consequence without comment below.  It is easy to see that one
can choose to have $a_{1}=a_{2}$, $b_{1}=b_{2}$ and $c_{1}=c_{2}$.
Assuming this, we have
\begin{align*}
p_{2}(x_{0},x_{2})=p_{1}(x_{0}+x_{2},x_{2}).
\end{align*}
Then we have
\begin{align*}
h(x_{2},x_{0})=\frac{p_{1}(x_{0}+x_{2},x_{2})}{x_{0}^{a_{1}}(x_{2}+x_{0})^{b_{1}}x_{2}^{c_{1}}}.
\end{align*} 
Considering
\begin{align*}
\frac{p_{1}(x_{1},x_{2})}{x_{0}^{a_{1}}x_{1}^{b_{1}}x_{2}^{c_{1}}}
\end{align*}
in place of $f(x_{0},x_{1},x_{2})$ in Proposition
\ref{prop:forwardenhanceddelta} now gives the result.
\end{proof}

\section{Vacuum-free vertex algebras}
\label{sec:vvalgebra}
\setcounter{equation}{0}

There are many variant definitions of vertex-type algebras.  For
instance, in \cite{LL} the authors recall Borcherds' notion of vertex
algebra \cite{B}, but using the formalism of the Jacobi identity (see
Definition 3.1.1 in \cite{LL}), which, among other things, lacks a
conformal vector, but does include a vacuum vector, the analogue of an
identity in a commutative associative algebra.  They purposely, for
reasons of expository clarity, redundantly state as axioms two
defining properties of the vacuum vector analogous to both the right
and left identity properties.  Indeed, they point out in Proposition
3.6.7 \cite{LL} that the vacuum property is redundant.  They further
show, in Remark 3.6.8 \cite{LL} that the creation property is not
redundant.  However, if we require as a separate axiom that the vertex
operator map be injective, this asymmetry between the analogues of
left and right identity disappears.  Indeed, injectivity follows from
the creation property, but by inserting this further redundancy into
the axioms it can be shown that either of the vacuum or creation
properties follows from the other when all the other axioms are
assumed (to see that the creation property is now redundant, see Remark
2.2.4 in \cite{FHL}).  Because of this, we shall use injectivity in
our statement of the axioms.

In the spirit of studying rings without identity, we may ask what we
would have if we removed the vacuum vector altogether.  Indeed, many
of the various versions of vertex-type algebra already form a
hierarchy of specialization as, for instance, vertex algebras
specialize to quasi- (or M\"obius) vertex algebras (cf. \cite{FHL}),
which in turn specialize to vertex operator algebras.  So there is
already ample precedent for a layered theory which we would be
extending.  For further justification, we note that the main axiom of
any version of vertex-type algebra is some form of the Jacobi
identity.  It is hoped that by removing the assumption of having a
vacuum vector, we avoid any premature distractions from this main
axiom in the early development of the theory and that this development
also shows more precisely the natural role that the vacuum vector
plays vis-\`a-vis the Jacobi identity when we specialize to that case.
With this as motivation, rather than any particular examples (although
see \cite{HL} and \cite{BD}) and further, since it turns out that we
can recover many results even in this pared-down setting, we proceed
to define a vacuum-free vertex algebra:
 
\begin{defi} \rm
A \it{vacuum-free vertex algebra} \rm is a vector space equipped,
first, with a linear map (the \it{vertex operator map}) $V
\otimes V \rightarrow V[[x,x^{-1}]]$, \rm or equivalently, a linear
map
\begin{align*}
Y(\,\cdot\,,x): \quad &V \, \rightarrow \, (\text{\rm End}V)[[x,x^{-1}]]\\
&v \, \mapsto \, Y(v,x)=\sum_{n \in \mathbb{Z}}v_{n}x^{-n-1}.
\end{align*}
We call $Y(v,x)$ the \it{vertex operator associated with} $v$.  \rm We
assume that the map
\begin{align*}
Y(\,\cdot\,,x): \quad &V \, \rightarrow \, (\text{\rm End}V)[[x,x^{-1}]]
\end{align*}
is injective.  We further assume that
\begin{align*}
Y(u,x)v \in V((x))
\end{align*}
for all $u,v \in V$.  Finally, we require that the
\it{Jacobi identity} \rm is satisfied:
\begin{align*}
x_{0}^{-1}\delta\left(\frac{x_{1}-x_{2}}{x_{0}}\right)Y(u,x_{1})Y(v,x_{2})&-
x_{0}^{-1}\delta\left(\frac{-x_{2}+x_{1}}{x_{0}}\right)Y(v,x_{2})Y(u,x_{1})\\
&=x_{1}^{-1}\delta\left(\frac{x_{2}+x_{0}}{x_{1}}\right)Y(Y(u,x_{0})v,x_{2}).
\end{align*}
\end{defi}
\begin{remark} \rm
It is well known that rngs can be embedded in rings.  Analogously, it
is not difficult to show that any vacuum-free vertex algebra can be
embedded in a vertex algebra.  We shall give a one line proof here
which uses a powerful representation theory result.  The adjoint
representation of any given vacuum-free vertex algebra yields a set of
mutually local vertex operators, and by, for instance, Theorem 5.5.18
in \cite{LL}, these operators generate a vertex algebra.
\end{remark}
We can immediately apply Proposition \ref{prop:elem} to obtain the
next two results.
\begin{prop}
\label{prop:vvalgrel}
Let $V$ be a vacuum-free vertex algebra.  For all $u,v,w \in V$, we have:
\begin{itemize}
\item
There exists some $m_{1} \geq 0$ such that
\begin{align*}
(x_{1}-x_{2})^{m_{1}}\left(Y(u,x_{1})Y(v,x_{2})-Y(v,x_{2})Y(u,x_{1})\right)=0
\end{align*}
(weak commutativity).
\item
There exists some $m_{2} \geq 0$ such that
\begin{align*}
(x_{0}+x_{2})^{m_{2}}\left(Y(u,x_{0}+x_{2})Y(v,x_{2})w-Y(Y(u,x_{0})v,x_{2})w\right)=0
\end{align*}
(weak associativity).
\item
There exists some $m_{3} \geq 0$ such that
\begin{align*}
(x_{1}-x_{0})^{m_{3}}\left(Y(v,-x_{0}+x_{1})Y(u,x_{1})w-Y(Y(u,x_{0})v,x_{1}-x_{0})w\right)=0
\end{align*}
(weak skew-associativity).
\item
There exist $p_{1}(x_{1},x_{2}) \in V[[x_{1},x_{2}]]$ and
$a_{1},b_{1},c_{1} \geq 0$ such that
\begin{align*}
Y(u,x_{1})Y(v,x_{2})w=\frac{p_{1}(x_{1},x_{2})}{(x_{1}-x_{2})^{a_{1}}x_{1}^{b_{1}}x_{2}^{c_{1}}}
\end{align*}
and
\begin{align*}
Y(v,x_{2})Y(u,x_{1})w=\frac{p_{1}(x_{1},x_{2})}{(-x_{2}+x_{1})^{a_{1}}x_{1}^{b_{1}}x_{2}^{c_{1}}}
\end{align*}
(formal commutativity).
\item
There exist $p_{2}(x_{0},x_{2}) \in V[[x_{0},x_{2}]]$ and
$a_{2},b_{2},c_{2} \geq 0$ such that
\begin{align*}
Y(u,x_{0}+x_{2})Y(v,x_{2})w=\frac{p_{2}(x_{0},x_{2})}{x_{0}^{a_{2}}
(x_{0}+x_{2})^{b_{2}}x_{2}^{c_{2}}}
\end{align*}
and
\begin{align*}
Y(Y(u,x_{0})v,x_{2})w=\frac{p_{2}(x_{0},x_{2})}{x_{0}^{a_{2}}(x_{2}+x_{0})^{b_{2}}x_{2}^{c_{2}}}
\end{align*}
(formal associativity).
\item
There exist $p_{3}(x_{0},x_{1}) \in
V[[x_{0},x_{1}]]$ and $a_{3},b_{3},c_{3} \geq 0$ such that
\begin{align*}
Y(v,-x_{0}+x_{1})Y(u,x_{1})w=\frac{p_{3}(x_{0},x_{1})}{x_{0}^{a_{3}}x_{1}^{b_{3}}
(-x_{0}+x_{1})^{c_{3}}}
\end{align*}
and
\begin{align*}
Y(Y(u,x_{0})v,x_{1}-x_{0})w=\frac{p_{3}(x_{0},x_{1})}{x_{0}^{a_{3}}x_{1}^{b_{3}}
(x_{1}-x_{0})^{c_{3}}}
\end{align*}
(formal skew-associativity).
\end{itemize}
\end{prop}
\begin{flushright} $\square$ \end{flushright}
\begin{remark} \rm
\label{rem:ms}
If we consider certain minimal values, which are ``obviously pole
clearing,'' that $m_{1},m_{2}$ and $m_{3}$ may be taken to be in the
above proposition, then it is easy to see that they could all be given
by the same function on suitable ordered pairs of vectors.  For instance,
$m_{1}(u,v)$ could be defined as the negative of the least integer
power of $x$ appearing in $Y(u,x)v$, whenever that power is negative,
and zero otherwise.  As a corollary to this, we see that we could
specify which vectors $m_{1},m_{2}$ and $m_{3}$ may depend on, using a
more refined statement, each one depending on only two vectors, but we
shall not need or want this refinement in this work (see, for
instance, Remarks 3.2.2 and 3.4.2 in \cite{LL}).
\end{remark}

\begin{prop}
\label{prop:two}
Any two of weak commutativity, weak associativity and weak
skew-associativity can replace the Jacobi identity in the definition
of the notion of vacuum-free vertex algebra.
\end{prop}
\begin{flushright} $\square$ \end{flushright}

\begin{remark} \rm
As we have seen, Proposition \ref{prop:two} really is a statement of 
formal calculus and does not need any special information from the
(vacuum-free) vertex algebra setting.
\end{remark}
\begin{remark} \rm
In previous treatments of vertex algebras (with a vacuum vector), as
far as the author is aware, weak skew-associativity has been neglected
since one can get by perfectly well with the other two weak properties.  
However, we shall see below that in some ways weak skew-associativity 
smoothes out some of the theory.
\end{remark}
\begin{remark} \rm
In the theory of vertex algebras with a vacuum vector, one obtains,
without extra assumption, a $\mathcal{D}$ (``derivation'') operator
which allows one to derive the skew-symmetry relation.  Since we do
not have a $\mathcal{D}$ operator at this stage we cannot get such a
relation, but we still have what we call ``vacuum-free
skew-symmetry,'' which is used occasionally without much comment in
other treatments.
\end{remark}
\begin{prop}
\label{prop:vfss}\bf{(vacuum-free skew-symmetry)} 
\it Let $V$ be a vacuum-free vertex algebra.  For all $u$ and $v \in V$, we have
\begin{align*}
Y(Y(u,x_{0})v,x_{2})=Y(Y(v, -x_{0})u,x_{2}+x_{0}).
\end{align*}
\end{prop} 
\begin{proof}
Notice that the left-hand side of the Jacobi identity is invariant
under the substitutions $(x_{0},x_{1},x_{2};u,v) \leftrightarrow
(-x_{0},x_{2},x_{1};v,u)$.  This means that we get the following
relation coming from the right-hand side:
\begin{align*}
x_{1}^{-1}\delta\left(\frac{x_{2}+x_{0}}{x_{1}}\right)Y(Y(u,x_{0})v,x_{2})
&=x_{2}^{-1}\delta\left(\frac{x_{1}-x_{0}}{x_{2}}\right)Y(Y(v,-x_{0})u,x_{1})\\
&=x_{1}^{-1}\delta\left(\frac{x_{2}+x_{0}}{x_{1}}\right)Y(Y(v,
-x_{0})u,x_{2}+x_{0}),
\end{align*}
so that taking the residue with respect to $x_{1}$ yields the result.
\end{proof}
It turns out that it is enough to know vacuum-free skew-symmetry
together with any single one of weak commutativity, weak
associativity, or weak skew-associativity in order to recover the
entire Jacobi identity.  We prove each of these equivalences, arguing
in a similar spirit to the proofs given in \cite{LL}, where the
authors used certain classical guides.  Our classical guides use
relationships found in commutative associative algebras without
identity element, but we shall not have all such relationships.
Indeed, the analogues we use are as follows:
\[ 
\begin{array}{lr}
a(bc)=b(ac) & \text{corresponds to weak commutativity;}\\ 
a(bc)=(ab)c & \text{corresponds to weak associativity;}\\ 
a(bc)=(ba)c & \text{corresponds to weak skew-associativity;}\\ 
\text{and each of} & \\
a(bc)=(bc)a,\,(ab)c=(ba)c, \, a(bc)=a(cb) & \text{corresponds
to vacuum-free skew symmetry}.
\end{array}
\]
\begin{remark} \rm
Actually, the third listed analogue for vacuum-free skew symmetry must
be derived from the other two, which is due to the fact that an iterate
appears explicitly in the statement of vacuum-free skew symmetry.
\end{remark}
\begin{prop}
\label{prop:weakassocplusskew}
Weak associativity together with vacuum-free skew-symmetry can replace
the Jacobi identity in the definition of the notion of vacuum-free
vertex algebra.
\end{prop}

\begin{proof}
We follow this analogue: $a(bc)=(ab)c=(ba)c$.  Let $V$ be a
vacuum-free vertex algebra.  We shall show that we get weak
skew-associativity, which will be enough.  In fact, for all $u,v \in
V$, there exists $l \geq 0$ such that
\begin{align*}
(x_{1}-x_{0})^{l}Y(u,-x_{0}+x_{1})Y(v,x_{1})w
&=(x_{1}-x_{0})^{l}Y(Y(u,-x_{0})v,x_{1})w\\
&=(x_{1}-x_{0})^{l}Y(Y(v,x_{0})u,x_{1}-x_{0})w.
\end{align*}
\end{proof}
\begin{prop}
\label{prop:weakskewassocplusskew}
Weak skew-associativity together with vacuum-free skew-symmetry can
replace the Jacobi identity in the definition of the notion of
vacuum-free vertex algebra.
\end{prop}

\begin{proof}
We follow this analogue: $a(bc)=(ba)c=(ab)c$.  Let $V$ be a
vacuum-free vertex algebra.  We shall show that we get weak
associativity which will be enough.  In fact, for all $u,v,w \in V$,
there exists $l \geq 0$ such that
\begin{align*}
(x_{0}+x_{2})^{l}Y(v,x_{0}+x_{2})Y(u,x_{2})w
&=(x_{0}+x_{2})^{l}Y(Y(u,-x_{0})v,x_{2}+x_{0})w\\
&=(x_{0}+x_{2})^{l}Y(Y(v,x_{0})u,x_{2})w.
\end{align*}
\end{proof}

\begin{prop}
\label{prop:weakcommplusskew}
Weak commutativity together with vacuum-free skew-symmetry can replace
the Jacobi identity in the definition of the notion of vacuum-free
vertex algebra.
\end{prop}

\begin{proof}
We follow this analogue: $a(bc)=(bc)a=(cb)a=a(cb)=c(ab)=(ab)c$.  Let
$V$ be a vacuum-free vertex algebra.  We shall show that we get weak
associativity, which will be enough.  For all $u,v,w \in V$, it is
easy to see that there exists $l \geq 0$ such that:
\begin{align*}
&(x_{0}+x_{2})^{l}Y(Y(u,x_{0}+x_{2})Y(v,x_{2})w,x_{3})\\
&=(x_{0}+x_{2})^{l}Y(Y(Y(v,x_{2})w,-x_{0}-x_{2})u,x_{3}+(x_{0}+x_{2}))\\
&=(x_{0}+x_{2})^{l}Y(Y(Y(w,-x_{2})v,-x_{0})u,x_{3}+(x_{0}+x_{2}))\\
&=(x_{0}+x_{2})^{l}Y(Y(u,x_{0})Y(w,-x_{2})v,x_{3}+x_{2})\\
&=(x_{0}+x_{2})^{l}Y(Y(w,-x_{2})Y(u,x_{0})v,x_{3}+x_{2})\\
&=(x_{0}+x_{2})^{l}Y(Y(Y(u,x_{0})v,x_{2})w,x_{3}),
\end{align*}
so that the result follows from the injectivity of the vertex operators.
\end{proof}
\begin{remark} \rm
We did not use the injectivity property of vertex operators in the
proofs of Propositions \ref{prop:weakassocplusskew} and
\ref{prop:weakskewassocplusskew}, but we did use it in the proof of
Proposition \ref{prop:weakcommplusskew}.
\end{remark}

\section{Modules}
\label{section:Module1}
\setcounter{equation}{0}

In this section we define the notion of module for a vacuum-free
vertex algebra and show a series of results paralleling those of
Section \ref{sec:vvalgebra}, with one significant exception.  We do
not have that (along with module skew-symmetry) module weak
commutativity can be a replacement axiom, although we do get that
module weak associativity and module weak skew-associativity may be
used as replacement axioms.  A heuristic reason for this may be seen
in the fact that in the commutative associative guides, ``$c$'' did
not remain in the rightmost position for the case of weak
commutativity, but it did for the other two cases.  We do have a
module weak skew-symmetry, which is in a contrast of sorts to the
situation with a vertex algebra, where one is tempted to ignore any
special skew-symmetric-like property of the module case, since the
underlying vertex algebra skew-symmetry is all that one needs.  We
have already done all the work for this section.  We state the results
for completeness.

\begin{defi} \rm
A \it{module} \rm for a vacuum-free vertex algebra is a vector space $W$
equipped with a linear map 
$V \otimes W \rightarrow W[[x,x^{-1}]]$, or
equivalently, a linear map
\begin{align*}
Y_{W}(\,\cdot\,,x):&V \, \rightarrow \, (\text{\rm End}W)[[x,x^{-1}]]\\
&v \, \mapsto \, Y_{W}(v,x)=\sum_{n \in \mathbb{Z}}v_{n}x^{-n-1}.
\end{align*}
We assume that
\begin{align*}
Y_{W}(u,x)w \in W((x))
\end{align*}
for all $u \in V$ and $w \in W$.  Then finally, we require that the
\it{Jacobi identity} \rm is satisfied:
\begin{align*}
x_{0}^{-1}\delta\left(\frac{x_{1}-x_{2}}{x_{0}}\right)Y_{W}(u,x_{1})Y_{W}(v,x_{2})w&-
x_{0}^{-1}\delta\left(\frac{-x_{2}+x_{1}}{x_{0}}\right)Y_{W}(v,x_{2})Y_{W}(u,x_{1})w\\
&=x_{1}^{-1}\delta\left(\frac{x_{2}+x_{0}}{x_{1}}\right)Y_{W}(Y(u,x_{0})v,x_{2})w.
\end{align*}
\end{defi}
\begin{remark} \rm
This definition adheres to the principle that a module should satisfy
all the defining properties of a vertex algebra that make sense, which
is essentially the \it{a priori} \rm motivation given on page 117 in
\cite{LL}.
\end{remark}
We present the results in parallel order to those in Section
\ref{sec:vvalgebra}.
\begin{prop}
\label{prop:vacfreemodstuff}
Let $V$ be a vacuum-free vertex algebra with module $W$.  For all $u,v
\in V,$ and $w \in W$, we have:
\begin{itemize}
\item There exists some $m_{1} \geq 0$ such that
\begin{align*}
(x_{1}-x_{2})^{m_{1}}\left(Y_{W}(u,x_{1})Y_{W}(v,x_{2})-Y_{W}(v,x_{2})Y_{W}(u,x_{1})\right)=0
\end{align*}
(weak commutativity).
\item
There exists some $m_{2} \geq 0$ such that
\begin{align*}
(x_{0}+x_{2})^{m_{2}}\left(Y_{W}(u,x_{0}+x_{2})Y_{W}(v,x_{2})w-Y_{W}(Y(u,x_{0})v,
x_{2})w\right)=0
\end{align*}
(weak associativity).
\item
There exists some $m_{3} \geq 0$ such that
\begin{align*}
(x_{1}-x_{0})^{m_{3}}\left(Y_{W}(v,-x_{0}+x_{1})Y_{W}(u,x_{1})w-Y_{W}(Y(u,x_{0})v,
x_{1}-x_{0})w\right)=0
\end{align*}
(weak skew-associativity).
\item
There exist $p_{1}(x_{1},x_{2}) \in
W[[x_{1},x_{2}]]$ and $a_{1},b_{1},c_{1} \geq 0$ such that
\begin{align*}
Y_{W}(u,x_{1})Y_{W}(v,x_{2})w=\frac{p_{1}(x_{1},x_{2})}{(x_{1}-x_{2})^
{a_{1}}x_{1}^{b_{1}}x_{2}^{c_{1}}}
\end{align*}
and
\begin{align*}
Y_{W}(v,x_{2})Y_{W}(u,x_{1})w=\frac{p_{1}(x_{1},x_{2})}{(-x_{2}+x_{1})^
{a_{1}}x_{1}^{b_{1}}x_{2}^{c_{1}}}
\end{align*}
(formal commutativity).
\item
There exist $p_{2}(x_{0},x_{2}) \in
W[[x_{0},x_{2}]]$ and $a_{2},b_{2},c_{2} \geq 0$ such that
\begin{align*}
Y_{W}(u,x_{0}+x_{2})Y_{W}(v,x_{2})w=\frac{p_{2}(x_{0},x_{2})}{x_{0}^
{a_{2}}(x_{0}+x_{2})^{b_{2}}x_{2}^{c_{2}}}
\end{align*}
and
\begin{align*}
Y_{W}(Y(u,x_{0})v,x_{2})w=\frac{p_{2}(x_{0},x_{2})}{x_{0}^{a_{2}}(x_{2}+x_{0})^
{b_{2}}x_{2}^{c_{2}}}
\end{align*}
(formal associativity).
\item
There exist $p_{3}(x_{0},x_{1}) \in
W[[x_{0},x_{1}]]$ and $a_{3},b_{3},c_{3} \geq 0$ such that
\begin{align*}
Y_{W}(v,-x_{0}+x_{1})Y_{W}(u,x_{1})w=\frac{p_{3}(x_{0},x_{1})}{x_{0}^
{a_{3}}x_{1}^{b_{3}}(-x_{0}+x_{1})^{c_{3}}}
\end{align*}
and
\begin{align*}
Y_{W}(Y(u,x_{0})v,x_{1}-x_{0})w=\frac{p_{3}(x_{0},x_{1})}{x_{0}^
{a_{3}}x_{1}^{b_{3}}(x_{1}-x_{0})^{c_{3}}}
\end{align*}
(formal skew-associativity).
\end{itemize}
\end{prop}
\begin{flushright} $\square$ \end{flushright}
\begin{remark} \rm
Concerning $m_{1},m_{2}$ and $m_{3}$, see Remark \ref{rem:ms}.
\end{remark}
\begin{prop}
\label{prop:twoformodJac}
Any two of weak commutativity, weak associativity and weak
skew-associativity (in the sense of Proposition
\ref{prop:vacfreemodstuff}) can replace the Jacobi identity in the
definition of the notion of module for a vacuum-free vertex algebra.
\end{prop}
\begin{flushright} $\square$ \end{flushright}

\begin{prop}
\label{prop:modvacfree}
\bf{(vacuum-free skew-symmetry)} \it For all $u$ and $v \in V$, a
vacuum-free vertex algebra with module $W$, we get the following
relation:
\begin{align*}
Y_{W}(Y(u,x_{0})v,x_{2})=Y_{W}(Y(v, -x_{0})u,x_{2}+x_{0}).
\end{align*}
\end{prop}  
\begin{proof}
The proof is essentially the same as for Proposition \ref{prop:vfss}.
\end{proof}
\begin{prop}
\label{prop:WAplusVFSSforJac}
In the definition of the notion of module for a vacuum-free vertex
algebra, the Jacobi identity can be replaced by weak associativity (in
the sense of Proposition \ref{prop:vacfreemodstuff}) together with
vacuum-free skew-symmetry (in the sense of Proposition
\ref{prop:modvacfree}).
\end{prop}
\begin{proof}
The proof is essentially the same as for Proposition
\ref{prop:weakassocplusskew}.
\end{proof}
\begin{prop}
\label{prop:WSAplusVFSSforJac}
In the definition of the notion of module for a vacuum-free vertex
algebra, the Jacobi identity can be replaced by weak
skew-associativity (in the sense of Proposition
\ref{prop:vacfreemodstuff}) together with vacuum-free skew-symmetry
(in the sense of Proposition \ref{prop:modvacfree}).
\end{prop}
\begin{proof}
The proof is essentially the same as for Proposition
\ref{prop:weakskewassocplusskew}.
\end{proof}

\section{Vertex algebras with vacuum}
\setcounter{equation}{0}

We now consider the case of a vacuum-free vertex algebra in which one
of the vertex operators acts as the identity.  That is, given a vertex
algebra $V$ we have a distinguished vector $\textbf{1} \in V$, which
we call the vacuum vector, with the property that $Y(\textbf{1},x)=1$,
where $1$ is the identity endomorphism of $V$.  Continuing the analogy
with commutative associative algebras from previous sections, we see
that such a vector is analogous to a left identity map.  We may then
wonder if there is some sort of right identity property.  Since
vacuum-free skew-symmetry switches the order of the two vectors, we
consider specializing to the case where one of the two vectors is the
vacuum vector.  We get
\begin{align}
\label{strongcreation}
Y(Y(u,x)\textbf{1},z)&=Y(Y(\textbf{1},-x)u,z+x) \nonumber\\
&=Y(u,z+x),
\end{align}  
{}from which it is easy to see that $Y(u,x)\textbf{1}$ must be a formal
power series in $x$.  Therefore we may set $x=0$, to check the
constant term, which gives us
\begin{align*}
Y(u_{-1}\textbf{1},z)=Y(u,z),
\end{align*}
which by injectivity gives
\begin{align*}
u_{-1}\textbf{1}=u.
\end{align*}
With this as motivation, we recall (\cite{B}; cf. \cite{LL}) the
definition of vertex algebra (with vacuum) and although the definition
contains redundancies we state both the left and right identity
properties for purposes of clarity and also since this is traditional.
\begin{defi} \rm
A \it{vertex algebra} \rm is a vacuum-free vertex algebra $V$ together
with a distinguished element $\textbf{1}$ satisfying the following
\it{vacuum property} \rm:
\begin{align*}
Y(\textbf{1},x)=1
\end{align*}
and \it{creation property} \rm:
\begin{align*}
Y(u,x)\textbf{1} \in V[[x]] \text{ and}\\
Y(u,0)\textbf{1}=u  \text{ for all }u \in V.
\end{align*}
\end{defi}
\begin{remark} \rm
\label{rem:minordef}
It shall be convenient for us to sometimes refer to ``minor axioms''
or separately ``minor properties'' of vertex algebras and their
modules.  For any object whose definition contains a
(vertex-algebraic) Jacobi identity, the ``minor axioms'' of that
object are all of the {\it explicitly} stated axioms except for the
Jacobi identity.  In the remainder of this paper, by ``minor
property'' of a vertex algebra we will mean any property from the
following list:
\begin{itemize}
\item Vacuum free skew-symmetry
\item Skew-symmetry
\item $\mathcal{D}$-bracket derivative property
\item $\mathcal{D}$-derivative property
\item Strong creation property,
\end{itemize}
some of which we have yet to recall, but soon shall.  
\end{remark}
Consider again the
following:
\begin{align*}
Y(Y(u,x)\textbf{1},z)&=Y(Y(\textbf{1},-x)u,z+x)\\
&=Y(u,z+x)\\
&=e^{x\frac{d}{dz}}Y(u,z).
\end{align*} 
Checking the first degree term in $x$ gives
\begin{align*}
Y(u_{-2}\textbf{1},z)=\frac{d}{dz}Y(u,z).
\end{align*} 
Thus vertex operators are closed under differentiation.  Furthermore,
we define the $\mathcal{D}$ operator.
\begin{defi} \rm
Given a vertex algebra, $V$, define $\mathcal{D} \in \text{\rm End}(V)$ by
\begin{align*}
\mathcal{D}v=v_{-2}\textbf{1}.
\end{align*}
\end{defi}
We may now write the \it{$\mathcal{D}$-derivative property} \rm:
\begin{align}
Y(\mathcal{D}u,z)=\frac{d}{dz}Y(u,z).
\end{align}
\begin{remark} \rm
\label{rem:minorfirst}
We note that the philosophy behind this proof of the
$\mathcal{D}$-derivative property is again a classical analogue coming
{}from commutative associative algebras with identity, namely, we used
as a guide the relation $a \cdot 1=1 \cdot a$.  In Proposition 3.1.18
\cite{LL}, the authors obtain the $\mathcal{D}$-derivative property in
a different fashion.  Their point of view was to observe that $v_{-2}
\textbf{1}$ is the component of a certain ``iterate'' and then to look
at the ``iterate formula,'' equation (3.1.11) \cite{LL} and ``slice
down'' to get the correct component.  We never need to make use of the
iterate formula in this work and, in fact, never use the word
``iterate'' except informally or to say that we won't use it.  As we
shall see (Propositions \ref{prop:leftrightid}, \ref{prop:minor1} and
\ref{prop:skewSymmetries} and Remarks \ref{rem:min} and
\ref{rem:minor}), the basic theory of the minor properties of a vertex
algebra, as well as some of the basic theory of those minor axioms
concerning the vacuum vector, can be handled entirely with vacuum-free
skew-symmetry without reference to the Jacobi identity or the iterate
formula.  However, we also note that the connection between the
$\mathcal{D}$-derivative property and properties similar to the
``iterate'' formula, namely, weak associativity and weak
skew-associativity, play an essential role in this treatment, as we
see in Propositions \ref{prop:WADDER}, \ref{prop:WSAgivesDDER} and
\ref{prop:modDder}, which are used to obtain Theorem
\ref{theorem:main}.
\end{remark}

In the introduction to Section \ref{sec:vvalgebra} we mentioned the
equivalence of the vacuum and creation properties provided we
separately state injectivity as an axiom.  In the introduction to this
section we saw how in the presence of vacuum-free skew-symmetry and
the other minor axioms, the vacuum property implies the creation
property.  We now show the converse.
\begin{prop}
\label{prop:leftrightid}
In the presence of vacuum-free skew-symmetry and the other minor
axioms of a vertex algebra, the vacuum property and the creation
property each imply the other.
\end{prop} 
\begin{proof}
We have already seen how the vacuum property implies the creation
property.  We begin the converse statement in a similar fashion by
specializing one of the vectors in the formula for vacuum-free
skew-symmetry to be $\textbf{1}$.  We get:
\begin{align}
\label{eq:creattovac}
Y(Y(\textbf{1},x)v,z)=Y(Y(v,-x)\textbf{1},z+x)\\
=e^{x\frac{d}{dz}}Y(Y(v,-x)\textbf{1},z).\nonumber
\end{align}
which by the first part of the creation property gives us that
$Y(\textbf{1},x)v$ is a power series in $x$.  Then extracting the
constant term in $x$ we have
\begin{align*}
Y(\textbf{1}_{-1}v,z)=Y(v_{-1}\textbf{1},z),
\end{align*}
which by the second part of the creation property and by injectivity
(which follows also from the second part of the creation property as
is usually argued) we have
\begin{align*}
\textbf{1}_{-1}v=v.
\end{align*}

We now need to show that $\textbf{1}_{-n}v=0$ for $n \geq 2$, or in
other words that $\frac{d}{dz}Y(\textbf{1},z)=0$.  It is tempting to
try and use the $\mathcal{D}$-derivative property to``peel off" the
``outer" $Y$ operator in (\ref{eq:creattovac}) but remember that we
used the vacuum property to get the $\mathcal{D}$-derivative property
so this is not available to us.  Instead, we try to imitate the
process of getting the $\mathcal{D}$-derivative property by checking
the linear term in $x$.  This gives us
\begin{align*}
Y(\textbf{1}_{-2}v,z)=-Y(v_{-2}\textbf{1},z)+\frac{d}{dz}Y(v,z).
\end{align*}
Then further specializing by setting $v=\textbf{1}$, we have
\begin{align}
\label{creationvacuum}
2Y(\textbf{1}_{-2}\textbf{1},z)=\frac{d}{dz}Y(\textbf{1},z).
\end{align}
Acting against $\textbf{1}$ and extracting the constant term in $z$
gives us, by the second part of the creation property, that
\begin{align*}
2\textbf{1}_{-2}\textbf{1}=\textbf{1}_{-2}\textbf{1},
\end{align*}
which gives us that $\textbf{1}_{-2}\textbf{1}=0$.  Then substituting back into
(\ref{creationvacuum}) gives
\begin{align*}
\frac{d}{dz}Y(\textbf{1},z)=0,
\end{align*}
which is what we needed to show.
\end{proof}
\begin{remark} \rm
\label{rem:min}
As was mentioned in the introduction to Section \ref{sec:vvalgebra},
Proposition 3.6.7 \cite{LL} shows that in the presence of the other
axioms the vacuum property may be derived and Remark 2.2.4 in
\cite{FHL} shows that in the presence of the other axioms (including
injectivity) the creation property may be derived.  Our proof in
Proposition \ref{prop:leftrightid} shares some common features with
the proofs in \cite{FHL} and \cite{LL}, but one difference that is
perhaps worth pointing out is that our assumption in each case was
weaker.  We needed only to assume vacuum-free skew-symmetry whereas
the proofs in both \cite{FHL} and \cite{LL} explicitly used the Jacobi
identity.  As was discussed in Remark \ref{rem:minorfirst} much of the
theory of the minor properties of a vertex algebra, as well as those
minor axioms concerning the vacuum vector, can be handled with only
the use of vacuum-free skew-symmetry (see Propositions
\ref{prop:minor1} and \ref{prop:skewSymmetries} and Remark
\ref{rem:minor} below).
\end{remark}

We now derive some of the standard ``minor properties.''  Taking the
exponential generating function of the higher derivatives and using
the $\mathcal{D}$-derivative property (and the formal Taylor theorem)
gives
\begin{align}
\label{eq:expDder}
Y(e^{x\mathcal{D}}u,z)=e^{x\frac{d}{dz}}Y(u,z)=Y(u,z+x),
\end{align}
which by (\ref{strongcreation}) gives
\begin{align*}
Y(e^{x\mathcal{D}}u,z)&=Y(u,z+x)\\
&=Y(Y(u,x)\textbf{1},z),
\end{align*}
which, by the injectivity of vertex operators, gives the \it{strong
creation property} \rm:
\begin{align}
Y(u,x)\textbf{1}=e^{x\mathcal{D}}u.
\end{align}

We again consider vacuum-free skew-symmetry in light of the
$\mathcal{D}$ operator, where we now have:
\begin{align*}
Y(Y(u,x)v,z)&=Y(Y(v,-x)u,z+x)\\
&=Y(e^{x\mathcal{D}}Y(v,-x)u,z),
\end{align*}  
which, by the injectivity of vertex operators, gives us \it{skew-symmetry}\rm:
\begin{align}
Y(u,x)v=e^{x\mathcal{D}}Y(v,-x)u.
\end{align}
We may take the derivative of this skew-symmetry formula to get:
\begin{align}
\label{skewD}
\frac{d}{dx}Y(u,x)v&=\mathcal{D}e^{x\mathcal{D}}Y(v,-x)u
+e^{x\mathcal{D}}\frac{d}{dx}Y(v,-x)u \nonumber\\
&=\mathcal{D}Y(u,x)v+e^{x\mathcal{D}}\frac{d}{dx}Y(v,-x)u \nonumber\\
&=\mathcal{D}Y(u,x)v-e^{x\mathcal{D}}Y(\mathcal{D}v,-x)u \nonumber \\
&=\mathcal{D}Y(u,x)v-Y(u,x)\mathcal{D}v,
\end{align}
where we used skew-symmetry to get the first, second and fourth
equalities and the $\mathcal{D}$-derivative property to get the third equality.
Observe that the last expression is a commutator, which in fact gives
us the \it{$\mathcal{D}$-bracket derivative property}\rm:
\begin{align}
[\mathcal{D},Y(u,x)]=\frac{d}{dx}Y(u,x).
\end{align}
Rearranging the terms of the $\mathcal{D}$-bracket formula makes it resemble a
product rule:
\begin{align}
\label{D-product1}
\mathcal{D}Y(u,x)v=\frac{d}{dx}Y(u,x)v+Y(u,x)\mathcal{D}v.
\end{align}
Of course, because of the $\mathcal{D}$-derivative property, we also have
\begin{align*}
[\mathcal{D},Y(u,x)]=Y(\mathcal{D}u,x).
\end{align*}
which, when the terms are rearranged, becomes
\begin{align}
\label{D-product2}
\mathcal{D}Y(u,x)v=Y(\mathcal{D}u,x)v+Y(u,x)\mathcal{D}v.
\end{align}
\begin{remark} \rm
\label{rem:umbral}
Whereas the $\mathcal{D}$-derivative property may be thought of as an
analogue of the power rule for differentiation, the
$\mathcal{D}$-bracket derivative property may be thought of as an
analogue of the product rule.  Indeed, while it is often messy to
check what properties look like component-wise, in the case of these
two properties the formulas are very familiar.  For the
$\mathcal{D}$-derivative property we have
\begin{align*}
(\mathcal{D}u)_{n}=-n(\mathcal{D}u)_{n-1}.
\end{align*}
For the $\mathcal{D}$-bracket derivative property (as rearranged in the form
given in (\ref{D-product2})) we have
\begin{align*}
\mathcal{D}(u_{n}v)=(\mathcal{D}u)_{n}v+u_{n}\mathcal{D}v.
\end{align*}
\end{remark}

Obviously, by the same reasoning as that which we gave for the automorphism
property, we have that (\ref{D-product2}) gives
\begin{align}
\label{D-Bracketexp1}
e^{z\mathcal{D}}Y(u,x)v=Y(e^{z\mathcal{D}}u,x)e^{z\mathcal{D}}v,
\end{align}
which by (\ref{eq:expDder}) gives us
\begin{align}
\label{D-Bracketexp2}
e^{z\mathcal{D}}Y(u,x)v=Y(u,x+z)e^{z\mathcal{D}}v,
\end{align}
which formula also follows directly from (\ref{D-product1}) and the
formal Taylor theorem, again using the same reasoning as the proof of
the automorphism property.

We have seen that skew-symmetry, together with the $\mathcal{D}$-derivative
property, gives us the $\mathcal{D}$-bracket derivative property.  Conversely,
we may also derive the $\mathcal{D}$-derivative property assuming only
skew-symmetry and the $\mathcal{D}$-bracket derivative property.  A careful
consideration of (\ref{skewD}), where our assumption is now that the
last line is equal to the first line, gives us:
\begin{align*}
Y(u,x)\mathcal{D}v&=-e^{x\mathcal{D}}\frac{d}{dx}Y(v,-x)u \Leftrightarrow\\
e^{-x\mathcal{D}}Y(u,x)\mathcal{D}v&=-\frac{d}{dx}Y(v,-x)u \Leftrightarrow\\
Y(\mathcal{D}v,-x)u&=-\frac{d}{dx}Y(v,-x)u,
\end{align*}
which is the $\mathcal{D}$-derivative property stated for $-x$.  
Given the $\mathcal{D}$-derivative property and skew-symmetry, we may also
recover vacuum-free skew-symmetry as can be seen by the following
calculation:
\begin{align*}
Y(u,x)v&=e^{x\mathcal{D}}Y(v,-x)u \Leftrightarrow\\
Y(Y(u,x)v,z)&=Y(e^{x\mathcal{D}}Y(v,-x)u,z) \Leftrightarrow\\
Y(Y(u,x)v,z)&=Y(Y(v,-x)u,z+x). 
\end{align*}
We summarize some of our implications in the next two propositions.
\begin{prop}
\label{prop:minor1}
In the presence of only the minor axioms of a vertex algebra, but
excluding the creation property, vacuum-free skew-symmetry implies the
strong creation property, skew-symmetry, the $\mathcal{D}$-derivative
property, and the $\mathcal{D}$-bracket derivative property.
\end{prop}
\begin{flushright} $\square$ \end{flushright}
\begin{prop}
\label{prop:skewSymmetries}
In the presence of the minor axioms of a vertex algebra, the following
are equivalent:
\begin{align*}
(i)& \text{ vacuum-free skew-symmetry}\\ (ii)& 
\text{ skew-symmetry
together with the $\mathcal{D}$-derivative property}\\ 
(iii)& \text{skew-symmetry together with the $\mathcal{D}$-bracket derivative
property}.
\end{align*}
\end{prop}
\begin{flushright} $\square$ \end{flushright}
\begin{remark} \rm
\label{rem:minor}
Our development of Proposition \ref{prop:minor1} largely parallels
portions of the material presented in Section 3.1 of \cite{LL}.
Perhaps the main difference is that our official proof of the
$\mathcal{D}$-derivative property is based on vacuum-free
skew-symmetry instead of the iterate formula (see equation 3.1.11 and
Proposition 3.1.18 \cite{LL}) and, more generally as well as more
roughly, that our point of view is that all of the minor properties
are due to vacuum-free skew-symmetry even without the full Jacobi
identity.
\end{remark}

Recall that we began this section by substituting the vacuum vector
into the formula for vacuum-free skew-symmetry.  We may pursue a
similar analysis with other formulas to further describe the
dependencies of weaker axioms on stronger ones.  

\begin{prop}
\label{prop:strongdep}
In the presence of the minor axioms of a vertex algebra, the strong
creation property follows from any single one of skew-symmetry, the
$\mathcal{D}$-bracket derivative property, or the
$\mathcal{D}$-derivative property.
\end{prop}
\begin{proof}
We first assume the $\mathcal{D}$-derivative property.  By the
$\mathcal{D}$-derivative property we have
\begin{align*}
Y(v,z+x)\textbf{1}=Y(e^{x\mathcal{D}}v,z)\textbf{1},
\end{align*}  
which by two applications of the creation property allows us to set
$z=0$ and calculate the right-hand side to get
\begin{align*}
Y(v,x)\textbf{1}=e^{x\mathcal{D}}v,
\end{align*}
which is the strong creation property.

We now assume the $\mathcal{D}$-bracket derivative property.  First, note that
by the vacuum property $\mathcal{D}\textbf{1}=\textbf{1}_{-2}\textbf{1}=0$, so
that $e^{x\mathcal{D}}\textbf{1}=\textbf{1}$.  Then by (\ref{D-Bracketexp2}) we
have
\begin{align*}
e^{x\mathcal{D}}Y(v,z)\textbf{1}=Y(v,z+x)\textbf{1},
\end{align*}
which again by two applications of the creation property gives the
strong creation property.  

Finally, we assume skew-symmetry.  We have
\begin{align*}
Y(u,x)\textbf{1}=e^{x\mathcal{D}}Y(\textbf{1},-x)u=e^{x\mathcal{D}}u,
\end{align*}
which is a third time, the strong creation property.  
\end{proof}

We continue to specialize our formulas by substituting in the vacuum
vector.  In the next proposition, we have the relation $a(b1)=b(a1)$
as a classical guide.

\begin{prop}
\label{prop:WCandDBgiveskew}
In the presence of the minor axioms of a vertex algebra, skew-symmetry
follows from weak commutativity together with the $\mathcal{D}$-bracket
derivative property.
\end{prop}
\begin{proof}
Let $V$ be a vertex algebra.  Let $u, v \in V$. By weak commutativity there
exists some $k \geq 0$ such that:
\begin{align*}
(x-z)^{k}Y(u,x)Y(v,z)\textbf{1}&=(x-z)^{k}Y(v,z)Y(u,x)\textbf{1}\\
&=(x-z)^{k}Y(v,z)e^{x\mathcal{D}}u\\
&=(x-z)^{k}e^{x\mathcal{D}}Y(v,z-x)u,
\end{align*}
and by the creation property we may set $z=0$ and cancel $x^{k}$ which
gives skew-symmetry.  
\end{proof}
\begin{remark} \rm
The statement and proof of Proposition \ref{prop:WCandDBgiveskew}
appeared as a special case of part 1 of Proposition 2.2.4 in
\cite{Li1}, where the greater generality in \cite{Li1} handled the
super-vertex algebra case.
\end{remark}

In fact, we have more:
\begin{prop}
In the presence of the minor axioms of a vertex
algebra\label{prop:WCJI}, weak commutativity together with the
$\mathcal{D}$-bracket derivative property are equivalent to the Jacobi
identity.
\end{prop}
\begin{proof}
The result follows from Proposition \ref{prop:WCandDBgiveskew},
Proposition \ref{prop:skewSymmetries} and Proposition
\ref{prop:weakcommplusskew}.
\end{proof}
\begin{remark} \rm
Proposition \ref{prop:WCJI} appeared in Theorem 3.5.1 \cite{LL}, where
Proposition \ref{prop:WCandDBgiveskew} was obtained during the course
of the proof.  Our development is similar to, but a variant of, the
proof presented in \cite{LL}.
\end{remark}

We next consider specializing the weak associativity property.  We have
\begin{prop}
\label{prop:WADDER}
In the presence of the minor axioms of a vertex algebra, weak
associativity implies the $\mathcal{D}$-derivative property.
\end{prop}
\begin{proof}
Let $V$ be a vertex algebra.  Let $u,w \in V$.  There exists $l \geq
0$ such that
\begin{align*}
(x_{0}+x_{2})^{l}Y(u,x_{0}+x_{2})Y(\textbf{1},x_{2})w&=(x_{0}+x_{2})^{l}
Y(Y(u,x_{0})\textbf{1},x_{2})w,
\end{align*}
and also such that the left-hand side of the equation is written in terms of
nonnegative powers of $(x_{0}+x_{2})$. Thus we have:
\begin{align*}
(x_{0}+x_{2})^{l}Y(u,x_{0}+x_{2})Y(\textbf{1},x_{2})w&=(x_{0}+x_{2})^{l}
Y(u,x_{2}+x_{0})Y(\textbf{1},x_{2})w\\
&=(x_{0}+x_{2})^{l}Y(u,x_{2}+x_{0})w.
\end{align*}
Then substituting this, we notice that $x_{0}$ is appropriately
truncated so that we can cancel $(x_{0}+x_{2})^{l}$ by multiplying by
$(x_{2}+x_{0})^{-l}$ and applying partial associativity.  This
gives us:
\begin{align*}
Y(u,x_{2}+x_{0})w&=Y(Y(u,x_{0})\textbf{1},x_{2})w \Leftrightarrow\\
e^{x_{0}\frac{d}{dx_{2}}}Y(u,x_{2})w&=Y(Y(u,x_{0})\textbf{1},x_{2})w.
\end{align*}
Looking at the linear term in $x_{0}$ gives the result.
\end{proof}
\begin{prop}
\label{prop:WAplusSCgivesDB}
In the presence of the minor axioms of a vertex algebra, weak
associativity together with the strong creation property imply the
$\mathcal{D}$-bracket derivative formula.
\end{prop}
\begin{proof}
Let $V$ be a vertex algebra.  Let $u,v \in V$.  There exists $l \geq
0$ such that:
\begin{align*}
(x_{0}+x_{2})^{l}Y(u,x_{0}+x_{2})Y(v,x_{2})\textbf{1}&=(x_{0}+x_{2})^{l}
Y(Y(u,x_{0})v,x_{2})\textbf{1}
\Leftrightarrow\\
(x_{0}+x_{2})^{l}Y(u,x_{0}+x_{2})e^{x_{2}\mathcal{D}}v&=(x_{0}+x_{2})^{l}
e^{x_{2}\mathcal{D}}Y(u,x_{0})v,
\end{align*}
which has only nonnegative powers of $x_{2}$ so that we may cancel
$(x_{0}+x_{2})^{l}$ which gives us the $\mathcal{D}$-bracket
derivative formula.
\end{proof}
\begin{remark} \rm
\label{rem:Li2a}
Propositions \ref{prop:WADDER} and \ref{prop:WAplusSCgivesDB} were
already essentially obtained as Proposition 2.6 in \cite{Li2}.  In
fact, Proposition 2.6 in \cite{Li2} was a stronger result, which shows
that the assumption of the strong creation property could have been
removed from Proposition \ref{prop:WAplusSCgivesDB}.  Comparing
arguments, we note that in the proof of Proposition
\ref{prop:WAplusSCgivesDB} we could have canceled $(x_{0}+x_{2})^{l}$
in the first line and extracted the coefficient of $x_{2}$ using the
creation property instead of extracting the coefficients of all powers
of $x_{2}$ using the strong creation property.  We would have arrived
at the ``unexponentiated'' form of the $\mathcal{D}$-bracket
derivative formula instead of the ``exponentiated'' form we did arrive
at, in parallel with the fact that the creation property is an
unexponentiated form of the strong creation property.  Corollary 2.7
in \cite{Li2} recovers the relevant ``exponentiated'' identities as a
consequence.
\end{remark}
\begin{prop}
\label{prop:WASSJI}
In the presence of the minor axioms of a vertex algebra, weak
associativity together with skew-symmetry are equivalent to the Jacobi
identity.
\end{prop}
\begin{proof}
By Proposition \ref{prop:strongdep} we have the strong creation
property and so by Proposition \ref{prop:WAplusSCgivesDB} we have the
$\mathcal{D}$-bracket derivative formula.  Then by Proposition
\ref{prop:skewSymmetries} we have vacuum-free skew-symmetry and so the
result follows from Proposition \ref{prop:weakassocplusskew}.
\end{proof}
We also have a slight variant proof of the last proposition.
\begin{proof}
By Proposition \ref{prop:WADDER} we have the $\mathcal{D}$-derivative
property so that by Proposition \ref{prop:skewSymmetries} we have
vacuum-free skew-symmetry and so the result follows from Proposition
\ref{prop:weakassocplusskew}.
\end{proof}
\begin{remark} \rm
Proposition \ref{prop:WASSJI} appeared in Theorem 3.6.1 in \cite{LL},
and Proposition \ref{prop:WAplusSCgivesDB} was essentially obtained in
the course of their proof.  Our present result generalizes more easily
to the module case.  In \cite{LL} the authors needed a further
argument, which they formulated in Theorem 3.6.3 \cite{LL}, in order
to obtain the corresponding result for a module.  In the course of the
proof of Theorem 3.6.3 in \cite{LL}, Proposition \ref{prop:WADDER} was
also obtained, though not stated separately.  It was our interest in
seeking an alternative to Theorem 3.6.3 in \cite{LL} that was the
original motivation for this paper.
\end{remark}

And finally, we consider weak skew-associativity.  
\begin{prop}
\label{prop:WSAgivesDDER}
In the presence of the minor axioms of a vertex algebra, weak
skew-associativity implies the $\mathcal{D}$-derivative property.
\end{prop}
\begin{proof} Let $V$ be a vertex algebra.  Let $u,w \in V$.  There exists $m \geq 0$ such that:
\begin{align*}
(x_{1}-x_{0})^{m}Y(u,x_{1})w&=(x_{1}-x_{0})^{m}Y(Y(u,x_{0})\textbf{1},x_{1}-x_{0})w
\Leftrightarrow\\ Y(u,x_{1})w&=Y(Y(u,x_{0})\textbf{1},x_{1}-x_{0})w
\Leftrightarrow\\
Y(u,x_{1})w&=e^{-x_{0}\frac{d}{dx_{1}}}Y(Y(u,x_{0})\textbf{1},x_{1})w
\Leftrightarrow\\
e^{x_{0}\frac{d}{dx_{1}}}Y(u,x_{1})w&=Y(Y(u,x_{0})\textbf{1},x_{1})w,
\end{align*}
where the cancellation of $(x_{1}-x_{0})^{m}$ was justified because
both sides had only nonnegative powers of $x_{0}$.  Then taking
coefficient of the first power of $x_{0}$ gives us the $\mathcal{D}$-derivative
property.
\end{proof}
We can substitute $\textbf{1}$ for still another vector to get:
\begin{prop}
\label{prop:wsssdb}
In the presence of the minor axioms of a vertex algebra, the following
are equivalent:
\begin{align*}
&(i)\text{ weak skew-associativity together with skew-symmetry}\\
&(ii)\text{ weak skew-associativity together with the $\mathcal{D}$-bracket
derivative property}.
\end{align*}
\end{prop}
\begin{proof}
Let $V$ be a vector space satisfying the relevant axioms.  Let $u,v
\in V$.  We assume $V$ satisfies weak skew-associativity.  By Proposition
\ref{prop:WSAgivesDDER} we have the $\mathcal{D}$-derivative property.
Then by Proposition \ref{prop:strongdep} we have the strong creation
property.  Then we have:
\begin{align*}
(x_{1}-x_{0})^{m}Y(v,-x_{0}+x_{1})Y(u,x_{1})\textbf{1}&=(x_{1}-x_{0})^{m}
Y(Y(u,x_{0})v,x_{1}-x_{0})\textbf{1}
\Leftrightarrow\\
(x_{1}-x_{0})^{m}Y(v,-x_{0}+x_{1})e^{x_{1}\mathcal{D}}u&=(x_{1}-x_{0})^{m}
e^{(x_{1}-x_{0})\mathcal{D}}Y(u,x_{0})v.
\end{align*}
Observing that both sides are truncated from below in $x_{1}$
appropriately we can cancel $(x_{1}-x_{0})^{m}$ from both sides to
get:
\begin{align*}
Y(v,-x_{0}+x_{1})e^{x_{1}\mathcal{D}}u&=e^{(x_{1}-x_{0})\mathcal{D}}Y(u,x_{0})v
\Leftrightarrow\\
e^{-x_{1}\mathcal{D}}Y(v,-x_{0}+x_{1})e^{x_{1}\mathcal{D}}u&=e^{-x_{0}\mathcal{D}}Y(u,x_{0})v,
\end{align*}
{}from which it is clear that either the $\mathcal{D}$-bracket
derivative property (in exponentiated form) or skew-symmetry each
implies the other.
\end{proof}

\begin{remark} \rm
\label{rem:Li2b}
We note that the argument in the proof of Proposition
\ref{prop:wsssdb} could have been changed to depend on only the
creation property instead of the strong creation property in a manner
similar to the changes discussed in Remark \ref{rem:Li2a}.
\end{remark}

We can now state two more replacement axioms for the Jacobi identity.
\begin{prop}
\label{prop:WSADBandSS}
In the presence of the minor axioms of a vertex algebra, weak
skew-associativity together with either single one of skew-symmetry or
the $\mathcal{D}$-bracket derivative property is equivalent to the
Jacobi identity.
\end{prop}
\begin{proof}
By Proposition \ref{prop:wsssdb} the two statements each follow from
the other.  Therefore it is enough to show the case where we assume
weak skew-associativity together with skew-symmetry.  By Proposition
\ref{prop:WSAgivesDDER} we have the $\mathcal{D}$-derivative property,
so that by Proposition \ref{prop:skewSymmetries} we have vacuum-free
skew-symmetry, which in turn gives us the result by Proposition
\ref{prop:weakskewassocplusskew}.
\end{proof}

\section{Modules for a vertex algebra with vacuum}
\setcounter{equation}{0}

In this section, we give the parallel results to those in Section
\ref{section:Module1}, where we now consider modules for a vertex
algebra (with vacuum).  Since any such module may also be viewed as a
module for a vacuum-free vertex algebra, most of the results carry
over without comment so we content ourselves with only discussing
certain new statements that we get.  Most importantly, we show that in
the notion of module for a vertex algebra, the Jacobi identity can
be replaced by either one (without the other) of weak associativity or
weak skew-associativity (in the sense of Proposition
\ref{prop:vacfreemodstuff}).
\begin{defi} \rm
A \it{module} \rm for a vertex algebra $V$ is a vector space $W$ which is a
vacuum-free module for $V$ when viewed as a vacuum-free vertex algebra
which further satisfies \it{the vacuum property} \rm:
\begin{align*}
Y_{W}(\textbf{1},x)=1,
\end{align*}
where $1$ is the identity operator on $W$.
\end{defi}
\begin{remark} \rm
We do not have an axiom for a module-type of creation property, and
this is not merely that we have chosen to remove any redundancy from
our axioms.  Indeed, our modules are really behaving as left modules
and so it does not make sense to have a right identity property, since
we cannot act on the vacuum vector.
\end{remark}
Following the proof of either Proposition \ref{prop:WADDER} or
Proposition \ref{prop:WSAgivesDDER}, we have the following
\it{$\mathcal{D}$-derivative property}\rm:
\begin{prop}
\label{prop:modDder2}
Let $W$ be a module for a vertex algebra $V$.  Then for any $v \in
V$, we have
\begin{align*}
Y_{W}(\mathcal{D}v,x)=\frac{d}{dx}Y_{W}(v,x).
\end{align*}
\end{prop}
\begin{flushright} $\square$ \end{flushright}

In fact, we have more, since the proofs of Proposition
\ref{prop:WADDER} and Proposition \ref{prop:WSAgivesDDER} imply the
following:
\begin{prop}
\label{prop:modDder}
In the presence of the minor axioms of a module for a vertex algebra,
either single one of weak associativity or weak skew-associativity
(each in the sense of Proposition \ref{prop:vacfreemodstuff}) implies
the $\mathcal{D}$-derivative property (in the sense of Proposition
\ref{prop:modDder2}).
\end{prop}
\begin{flushright} $\square$ \end{flushright}

We now conclude with the main result of this paper.  We have already
done all the work.  The result for weak associativity was obtained in
Theorem 4.4.5 in \cite{LL}.  It is this result as regards weak
associativity (or more precisely, an easy corollary of it, Corollary
4.4.7 \cite{LL}) which entered into the proof in \cite{LL} showing the
equivalence of the notion of representation of a vertex algebra with
the notion of a vertex algebra module (see Theorem 5.3.15 in
\cite{LL}).  We have seen that the Jacobi identity may be replaced by
weak associativity together with weak skew-associativity (in the sense
of Proposition \ref{prop:vacfreemodstuff}).  In fact, by using the
algebra skew-symmetry, which we have ``for free,'' we obtain that
either one of the two is enough.  It is shown in \cite{LL} that the
same is not true for weak commutativity (see Remark 4.4.6 in
\cite{LL}).
\begin{theorem}
\label{theorem:main}
In the presence of the minor axioms of module for a vertex algebra,
either single one of weak associativity or weak skew-associativity
(each in the sense of Proposition \ref{prop:vacfreemodstuff}) is
equivalent to the Jacobi identity.
\end{theorem}
\begin{proof}    
By Proposition \ref{prop:modDder} we have the $\mathcal{D}$-derivative
property, and following the proof of Proposition
\ref{prop:skewSymmetries} we have vacuum-free skew-symmetry (in the
sense of Proposition \ref{prop:modvacfree}).  Thus Proposition
\ref{prop:WAplusVFSSforJac} and Proposition
\ref{prop:WSAplusVFSSforJac} give the result.
\end{proof}

\noindent {\small \sc Department of Mathematics, Rutgers University,
Piscataway, NJ 08854} 
\\ {\em E--mail
address}: thomasro@math.rutgers.edu
\end{document}